\documentclass[11pt]{amsart}
\usepackage{amssymb}
\usepackage{amsmath}
\usepackage{amsfonts}
\usepackage{graphicx}

\usepackage[total={17cm,22cm},top=2.5cm, left=2.3cm]{geometry}
\usepackage{hyperref}
    \usepackage{aeguill}
    \usepackage{type1cm}

\usepackage[shortlabels]{enumitem}

\newtheorem{thm}{Theorem}

\newtheorem{lem}[thm]{Lemma}

\newtheorem{rem}[thm]{Remark}

\newtheorem{defn}[thm]{Definition}

\usepackage{amsxtra}
\usepackage{eucal}
\usepackage{mathrsfs}
\usepackage{longtable}
\usepackage{caption}

\usepackage{hyperref}
    \usepackage{aeguill}
    \usepackage{type1cm}

\newcommand{\N}{\mathbb{N}}

\newcommand{\R}{\mathbb{R}}

\usepackage[usenames,dvipsnames]{color}

\usepackage[dvipsnames]{xcolor}

\begin{document}

\title{Some remarks about The Morse-Sard theorem and approximate differentiability}

\author{Daniel Azagra}
\address{ICMAT (CSIC-UAM-UC3-UCM), Departamento de An{\'a}lisis Matem{\'a}tico,
Facultad Ciencias Matem{\'a}ticas, Universidad Complutense, 28040, Madrid, Spain }
\email{azagra@mat.ucm.es}

\author{Miguel Garc\'ia-Bravo}
\address{ICMAT (CSIC-UAM-UC3-UCM), Calle Nicol\'as Cabrera 13-15.
28049 Madrid SPAIN}
\email{miguel.garcia@icmat.es}

\date{April 25, 2017}


\begin{abstract}
Let $n, m$ be positive integers, $n\geq m$. We make several remarks on the relationship between approximate differentiability of higher order and Morse-Sard properties. For instance, among other things we show that if a function $f:\mathbb{R}^n\to\mathbb{R}^m$ is locally Lipschitz and is approximately differentiable of order $i$ almost everywhere with respect to the Hausdorff measure $\mathcal{H}^{i+m-2}$, for every $i=2, \dots, n-m+1$, then $f$ has the Morse-Sard property (that is to say, the image of the critical set of $f$ is null with respect to the Lebesgue measure in $\R^m$).
\end{abstract}

\maketitle

\section{Introduction}

The Morse-Sard theorem \cite{Mo, Sa} states that if $f:\R^n\to\R^m$ is of class $C^k$, where $k=\max\{n-m+1, 1\}$, then the set of critical values of $f$ has measure zero in $\R^m$. A celebrated example of Whitney's \cite{Whitney1} shows that this classical result is sharp within the classes of functions $C^j$. Given the crucial applications of the Morse-Sard theorem in several branches of mathematics, it is nonetheless natural and useful to try and refine the Morse-Sard theorem for other classes of functions, and by now there is a rich literature in this line of work. We cannot mention and comment on all of the important contributions generated by this problem; instead we shall content ourselves with referring the reader to \cite{Yom, No, Ba, Pa, Fi, Bo1, Bo2, KK1, KK2, HaZi, Hkk, Az2} and the references therein.

In this note we will show how, by combining some of the strategies and tools which are common to several of these with the idea of the proof of \cite[Theorem 1]{Liu} and an induction argument, one can obtain the following result: let $n\geq m$ and $f:\mathbb{R}^n\to\mathbb{R}^m$ be a Borel function. Suppose that $f$ is approximately differentiable of order $1$ at $\mathcal{H}^m-$almost every point and satisfies
\begin{enumerate}\item[(a)] $\text{ap}\limsup_{y\rightarrow x}{\vert f(y)-f(x)\vert\over \vert y-x\vert}<+\infty$ for all $x\in\mathbb{R}^n\setminus N_0$, where $N_0$ is a countable set, and
\item[(b)] $\text{ap}\lim_{y\rightarrow x}{\vert f(y)-f(x)-\cdots {F_i(x)\over i!}(y-x)^i\vert\over \vert y-x\vert^{i}}=0$ for all $i=2,\dots, n-m+1$ and for all $x\in\mathbb{R}^n\setminus N_i$, where each set $N_i$ is $(i+m-2)-$sigmafinite and the coefficients $F_i(x)$ are Borel functions,
\end{enumerate}
then $f$ has the Morse-Sard property (that is to say, the image of the critical set of $f$ is null with respect to the Lebesgue measure in $\R^m$). See Theorem \ref{p} in Section 3 below for a precise statement and proof. In Theorem \ref{i} we are able to dispense with the condition about the Borel measurability of the functions but we must strengthen conditions $(b)$ above by replacing the $(s)-$sigmafinite exceptional sets with countably $(\mathcal{H}^s,s)$ rectifiable sets of certain classes $C^k$. In Theorem \ref{o} we provide an interesting variant of this result. See Sections 2 and 3 for auxiliary results and definitions. Theorem \ref{i} and Theorem \ref{p} generalize the versions of the Morse-Sard theorem provided by Bates's theorem and the Appendix of \cite{Az2}, and are not stronger, nor weaker, than the versions of \cite{Bo2, KK1, KK2} for $BV_n$ or Sobolev functions with smaller exponents; see Section 4 below for examples and further comments.

\section{Preliminaries}

Recall that a modulus of continuity is a concave, increasing function $\omega:[0,\infty)\to [0, \infty)$ such that $\omega(0^{+})=0$. Given a positive integer $k$, for a fixed modulus of continuity $\omega$ the class $C^{k,\omega}(\R^n; \R)$ is defined as the set of functions which are $k$ times differentiable and its partial derivatives are uniformly continuous with modulus of continuity $\omega$. In the particular case $\omega (s)=s^{t}$ for some $t\in (0,1]$ we will write $C^{k,t}(\mathbb{R}^n;\mathbb{R})$.\\

A fundamental tool in our proofs of Theorem \ref{i} and Theorem \ref{o} will be the following version of the Whitney extension theorem, see \cite{Glaeser, Mal, St}.

\begin{thm}[Uniform version of Whitney Extension Theorem]\label{whi}
Let $\omega$ be a modulus of continuity.
Let $C$ be a subset of $\R^n$ and $\{ f_\alpha \}_{|\alpha|\leq k}$ be a family of real valued functions defined on $C$ satisfying
\begin{equation} \label{derivadas de whitney}
f_\alpha(x)= \sum_{|\beta| \leq k-|\alpha|} \frac{f_{\alpha+\beta}(y)}{\beta !} (x-y)^\beta
 + R_\alpha(x,y)
\end{equation}
for all $x,y \in  C$ and all multi-indices $\alpha$ with $|\alpha| \leq k.$ Suppose that for some constant $M>0$ we have
 \begin{equation}\label{condicion whitney}
 |f_\alpha(x)| \leq M, \textrm{ and } \quad |R_\alpha(x,y)|\leq M |x-y|^{k-|\alpha|}\omega(|x-y|) \quad\ \text{for all} \quad x,y \in C \textrm{ and all } |\alpha| \leq k.
 \end{equation}
 Then there exists a function $F:\R^n \to \R$ such that:
 \begin{itemize}
  \item[(i)] $F\in C^{k,\omega}(\R^n; \R).$
 \item[(ii)] $D^\alpha F = f_\alpha$ on $C$ for all $|\alpha| \leq k$.
 \end{itemize}
\end{thm}

Our notation with multi-indices is the standard one (see e.g. \cite[p. 2]{Zie2}).
This version of the Whitney extension theorem is usually stated for {\em closed} subsets $C$ of $\mathbb{R}^n$, but it is easily checked that Theorem \ref{whi} also holds for arbitrary subsets $C\subset\mathbb{R}^n$, because a modification of the usual argument showing that an uniformly continuous function defined on a set $D$ has a unique uniformly continuous extension (with the same modulus of continuity) to the closure $\overline{D}$ of $D$, together with conditions \eqref{derivadas de whitney} and \eqref{condicion whitney}, imply that if $C$ is not closed then the functions $f_{\alpha}$ have unique extensions to $\overline{C}$ that also satisfy \eqref{derivadas de whitney} and \eqref{condicion whitney} on $\overline{C}$. The theorem also remains true if we replace the target space $\R$ with $\R^m$, as one can apply the above result to the coordinate functions of $f=(f^1,\dots, f^m)$. In our proofs of Theorem \ref{i} and Theorem \ref{o} we will use this version of the Whitney extension theorem in the particular instances of $\omega(s)=s$ (thus obtaining extensions of class $C^{k, 1}$), or $\omega(s)=s^{t}$, with $0<t<1$ (in which case we will have extensions belonging to the H\"older differentiability classes $C^{k,t}$).

We will also use Whitney's original theorem for $C^k$, which we next restate for the reader's convenience.

\begin{thm}[Whitney Extension Theorem]\label{whi2}
Let $C\subset\mathbb{R}^{n}$ be closed. A necessary and sufficient condition, for a function $f: C\to\R$ and a family of functions
$\{ f_\alpha \}_{|\alpha|\leq k}$ defined on $C$ satisfying $f=f_0$ and
$$
f_\alpha(x)= \sum_{|\beta| \leq k-|\alpha|} \frac{f_{\alpha+\beta}(y)}{\beta !} (x-y)^\beta
 + R_\alpha(x,y)
$$
for all $x,y \in  C$ and all multi-indices $\alpha$ with $|\alpha| \leq k$, to admit a $C^k$ extension $F$ to all of $\R^n$ such that $D^\alpha F =f_\alpha$ on $C$ for all $|\alpha| \leq k$, is that
$$
  \lim_{|x-y| \to 0} \frac{R_\alpha(x,y)}{|x-y|^{m-|\alpha|}} =0 \eqno(W^{m})
$$
uniformly on compact subsets of $C$, for every $|\alpha| \leq k$.
\end{thm}

Let us set some other notation. We denote by $\mathcal{L}^n(E)$ the outer Lebesgue measure of a set $E\subseteq \mathbb{R}^n$. For $s\geq 0$, the $s-$dimensional Hausdorff measure is denoted by $\mathcal{H}^s$, and the $s-$dimensional Hausdorff content by $\mathcal{H}^{s}_{\infty}$. Recall that for any subset $E$ of $\mathbb{R}^n$ we have, by definition,
$$\mathcal{H}^s(E)=\lim_{\delta\searrow 0}\mathcal{H}^{s}_{\delta}(E)=\sup_{\delta >0}\mathcal{H}^{s}_{\delta}(E),$$
where for each $0<\delta\leq \infty$,
$$\mathcal{H}^{s}_{\delta}(E)=\inf\left\lbrace \sum^{\infty}_{i=1}(\text{diam}\, F_i)^s:\, \text{diam}\, F_i\leq \delta,\,E\subseteq \bigcup^{\infty}_{i=1} F_i\right\rbrace .$$
It is well known that the measures $\mathcal{H}^n$,  $\mathcal{H}^{n}_{\infty}$ and $\mathcal{L}^n$ are equivalent on $\R^n$, and that $\mathcal{H}^{s}$ and $\mathcal{H}^{s}_{\infty}$ have the same null sets. Another important fact about the Hausdorff measure $\mathcal{H}^s$ is that it is Borel regular (see e.g. \cite{EvGa}).

We will also say that a set is {\em $s-$sigmafinite} if it can be written as a countable union of sets with finite $\mathcal{H}^s-$measure. 

A set $N\subseteq\mathbb{R}^n$ is called {\em countably $(\mathcal{H}^s,s)$ rectifiable of class $C^{k}$} if and only if there exist countably many $s-$dimensional submanifolds $A_j$ of class $C^{k}$ such that $\mathcal{H}^s(N\setminus \bigcup^{\infty}_{j=1}A_j)=0$. This notion has been introduced in \cite{Anz}.

We will need some more definitions. A function $p:\mathbb{R}^n\rightarrow\mathbb{R}^m$ is said to be a polynomial of degree $k$ centered at the point $x\in\mathbb{R}^n$ if it is written in the form
$$p(x;y)=\sum_{\vert \alpha\vert\leq k}{p_{\alpha}(x)\over \alpha !}(y-x)^{\alpha},$$
where each $p_{\alpha}(x)=(p^1_{\alpha}(x),\dots,p^m_{\alpha})\in\R^m$.

A function $f:\mathbb{R}^n\rightarrow\mathbb{R}^m$ is said to be {\em approximately differentiable} of order $k$ at $x\in\mathbb{R}^n$ if there is a polynomial $p_k(x;y)$, centered at $x$, and of degree at most $k$, such that $$\text{ap}\lim_{y\rightarrow x}{\left| f(y)-p_k(x;y)\right|\over \vert y-x\vert^k}=0.$$
On the other hand $f$ will be said to have an {\em approximate $(k-1)-$Taylor polynomial} at $x$ if there is a polynomial $p_{k-1}(x;y)$, centered at $x$, and of degree at most $k-1$, such that
$$\text{ap}\limsup_{y\rightarrow x}{\left| f(y)-p_{k-1}(x;y)\right|\over \vert y-x\vert^k}<+\infty.$$
We recall that $\text{ap}\lim_{y\rightarrow x}f(y)=l$ means that for every $\varepsilon>0$,
$$\lim_{r\rightarrow 0} {\mathcal{L}^n\left( B(x,r)\cap\left\lbrace y\in\mathbb{R}^n:\,\vert f(y)-l\vert\geq\varepsilon\right\rbrace \right) \over \mathcal{L}^n(B(x,r))}=0,$$
and that $\text{ap}\limsup_{y\rightarrow x}f(y)$ is the infimum of all those $\lambda\in\mathbb{R}$ such that
$$\lim_{r\rightarrow 0} {\mathcal{L}^n\left( B(x,r)\cap\left\lbrace y\in\mathbb{R}^n:\, f(y)\geq\lambda\right\rbrace \right) \over \mathcal{L}^n(B(x,r))}=0.$$
Observe that if a function $f$ is of class $C^k$ then in particular $f$ has a Taylor expansion of order $k$ at $x$, and therefore $f$ is approximately differentiable of order $k$ at $x$, with corresponding Taylor polynomial
$$p_k(x;y)=\sum_{\vert \alpha\vert\leq k}{D^{\alpha}f(x)\over \alpha !}(y-x)^{\alpha}.$$
($D^{\alpha}f(x)=(D^{\alpha}f^1(x),\dots,D^{\alpha}f^m(x))$).

It should be noted that if $f$ is approximately differentiable of order $k$ (or has an approximate $(k-1)-$Taylor polynomial) at $x$ then the corresponding polynomial $p_{k}(x;y)$ (or $p_{k-1}(x;y)$) of degree at most $k$ (or $k-1$) is unique. Actually all the usual rules about differentiability of sums, products and quotients of functions apply to approximate differentiable functions as well. 
 
Now we can change our notation and express the unique polynomial $p_{k}(x;y)$ by
\begin{equation}\label{aprox}
p_{k}(x;y)=\sum_{\vert\alpha\vert\leq k}{f_{\alpha}(x)\over \alpha!}(y-x)^{\alpha}.
\end{equation}
In particular $f_0(x)=f(x)$.

Observe that using the notation $f_{\alpha}$ does not by any means imply that there exists a derivative $D^{\alpha}f$ in the usual sense, nor that $f_{\alpha}(x)=D^{\alpha}f(x)$ even if $D^{\alpha}f(x)$ exists.

From now on, every time we say a function or a set is measurable, and unless we specify the measure, we will mean it with respect to the Lebesgue measure. 

In the proofs of Theorem \ref{i} and Theorem \ref{o} we will make heavy use of the following Lemma, which is an easy consequence of an argument of Liu and Tai in the proof of \cite[Theorem 1]{Liu}. For completeness, and for the reader's convenience, we provide a detailed argument.

\begin{lem}\label{u}
Let $f:\mathbb{R}^n\rightarrow\mathbb{R}^m$ be a measurable function, $k$ a positive integer and $N$ a subset of $\mathbb{R}^n$. Consider the following statements.
\begin{enumerate}[(i)]
\item $f$ is approximately differentiable of order $k$ for all $x\in\mathbb{R}^n\setminus N$.
\item $f$ has an approximate $(k-1)-$Taylor polynomial for all $x\in\mathbb{R}^n\setminus N$.
\item There exists a decomposition
$$\mathbb{R}^n=\bigcup^{\infty}_{j=1}B_j\cup N,$$
such that for each $j\in\mathbb{N}$ there is a function $g_j\in C^{k-1,1}(\mathbb{R}^n;\mathbb{R}^m)$ with $f_{\alpha}(x)=D^{\alpha}g_j(x)$ for all $x\in B_j$ and $\vert \alpha\vert \leq k-1$.
\end{enumerate}
Then we have that $(i)\Rightarrow (ii) \Rightarrow (iii)$.
\end{lem}
\begin{proof}
We will need to use the following.
\begin{lem}[De Giorgi]\label{ua}
Let $E$ be a measurable subset of the ball $B(x,r)\subset \R^n$ such that $\mathcal{L}^n(E)\geq Ar^n$ for some constant $A>0$. Then for each positive integer $k$ there is a positive constant $C$, depending only on $n,k$ and $A$, such that
$$\left|D^{\alpha}p(x)\right|\leq {C\over r^{n+|\alpha|}}\int_{E}\vert p(y)\vert\,dy$$
for all polynomials $p$ of degree at most $k$ and all multi-indices $|\alpha|\leq k$.
\end{lem}
See \cite{Ca} for a proof of De Giorgi's lemma. 

$(i)\Rightarrow (ii):$ This implication is straightforward. The same points for which $(i)$ holds make $(ii)$ true. Indeed, if a polynomial that gives $(i)$ centered at some $x$ is
$$p_k(x;y)=\sum_{\vert\alpha\vert\leq k}{p_{\alpha}(x)\over \alpha!}(y-x)^{\alpha},$$
we take $p_{k-1}(x;y)=\sum_{\vert\alpha\vert\leq k-1}{p_{\alpha}(x)\over \alpha!}(y-x)^{\alpha}$ and let $\lambda=1+\sum_{\vert \alpha\vert =k}{\vert p_{\alpha}(x)\vert\over\alpha!}<\infty$. Since
$${\vert f(y)-p_{k-1}(x;y)\vert\over \vert y-x\vert^k}\leq {\vert f(y)-p_k(x;y)\vert\over \vert y-x\vert^k} + \sum_{\vert \alpha\vert =k}{\vert p_{\alpha}(x)\vert\over\alpha!}, $$
we have that
$$\left\lbrace y\in\mathbb{R}^n:\, \vert f(y)-p_k(x;y)\vert\leq \vert y-x\vert^k\right\rbrace \subseteq \left\lbrace y\in\mathbb{R}^n:\, \vert f(y)-p_{k-1}(x;y)\vert\leq \lambda\vert y-x\vert^k\right\rbrace,$$
hence by our hypothesis $(i)$
$${\mathcal{L}^n\left( B(x,r)\cap\left\lbrace y\in\mathbb{R}^n:\, {\vert f(y)-p_{k-1}(x;y)\vert\over \vert y-x\vert^k}\leq \lambda\right\rbrace \right) \over \mathcal{L}^n(B(x,r))}\geq {\mathcal{L}^n\left( B(x,r)\cap\left\lbrace y\in\mathbb{R}^n:\, {\vert f(y)-p_k(x;y)\vert\over \vert y-x\vert^k}\leq 1\right\rbrace \right) \over \mathcal{L}^n(B(x,r))}\xrightarrow{r\rightarrow 0} 1, $$
which implies that
$$\text{ap}\limsup_{y\rightarrow x}{\vert f(y)-p_{k-1}(x;y)\vert\over \vert y-x\vert^k}\leq \lambda<+\infty.$$

$(ii)\Rightarrow (iii):$ Recall that the approximate $(k-1)-$Taylor polynomials are unique so we may use the notation of \eqref{aprox}.
The idea of the proof consists in splitting $\mathbb{R}^n\setminus N$ into a countable union of sets $\left\lbrace B_j\right\rbrace _{j\geq 1}$, on each of which, with the help of De Giorgi's lemma, we can apply Theorem \ref{whi}. We will show that for each $j\in\mathbb{N}$, 
$$ \left\{ \begin{array}{l} 
\vert f_{\alpha}(y)-D^{\alpha}p_{k-1}(x;y)\vert\leq M\vert x-y\vert^{k-\vert\alpha\vert},\;\; \forall x,y\in B_j,\; \vert y-x\vert \leq{1\over j},\; \vert \alpha\vert\leq k-1 \vspace{0.3cm} \\  \vert f_{\alpha}(x)\vert\leq j,\;\;  \forall x\in B_j    
                                                \end{array}
\right.  $$
where $p_{k-1}(x;y)$ is the polynomial of degree at most $k-1$ that gives $(ii)$ and $M$ is a constant (to be fixed later on) that depends only on $n$, $k$ and $j$.

Let us define
$$  \begin{array}{l} 
\rho:=\displaystyle{{\mathcal{L}^n(B(x,\vert y-x\vert)\cap B(y,\vert y -x\vert))\over \vert y-x\vert^n}},\;\;x,y\in\mathbb{R}^n,\,x\neq y\vspace{0.2cm}\\
W_{j}(x;r):=B(x,r)\setminus\left\lbrace y\in\mathbb{R}^n:\,\vert f(y)-p_{k-1}(x;y)\vert\leq j\vert y-x\vert^k\right\rbrace ,\;\; x\in\mathbb{R}^n,\;r>0,\;j\in\mathbb{N} \vspace{0.2cm}\\
B_j:=\left\lbrace x\in\mathbb{R}^n:\,\mathcal{L}^n(W_j(x;r))\leq\rho{r^n\over 4},\;r\leq {1\over j}\right\rbrace \cap\left\lbrace x\in\mathbb{R}^n:\;\vert f_{\alpha}(x)\vert\leq j,\,\vert\alpha\vert \leq k-1\right\rbrace.  
\end{array}
$$             
Note that $\rho$ only depends on $n$. Since $f$ is measurable  we have that $W_j(x,r)$ are measurable sets. It is immediately checked that $B_j$ is an increasing sequence of sets and $$\bigcup^{\infty}_{j=1}B_j=\mathbb{R}^n\setminus N.$$ For us it will not be important that the sets $B_j$ and the coefficients $f_{\alpha}$ are measurable, although they are indeed so (see Liu-Tai's proof for the delicate induction argument that allows one to show this).

Now, given $j\in\mathbb{N}$, consider two different points $x,y\in B_j$ with $\vert y-x\vert\leq {1\over j}$, and for $r=\vert y-x\vert$ let
$$S(x,y;r,j):=\left[ B(x,r)\cap B(y,r)\right] \setminus\left[ W_j(x;r)\cup W_j(y;r)\right],$$
which is measurable. Moreover,
$$\mathcal{L}^n(S(x,y;r,j))\geq \mathcal{L}^n(B(x,r)\cap B(y,r))-\mathcal{L}^n(W_j(x;r))-\mathcal{L}^n(W_j(y;r))\geq\rho {r^n\over 2}>0.$$
If we take $z\in S(x,y;r,j)$ then we have for $q(z)=p_{k-1}(y;z)-p_{k-1}(x;z)$ the estimate
$$\vert q(z)\vert\leq\vert p_{k-1}(x;z)-f(z)\vert+\vert f(z)-p_{k-1}(y;z)\vert\leq (\vert z-x\vert^k+\vert y-z\vert^k)\leq 2jr^k.$$
We now apply Lemma 1 with $E=S(x,y;r,j)$ to obtain that for each multi-index $|\alpha|\leq k-1$,
$$\vert D^{\alpha}q(y)\vert=\vert f_{\alpha}(y)-D^{\alpha}p_{k-1}(x;y)\vert\leq {C\over r^{n+\vert\alpha\vert}}\int_{S(x,y;r,j)}\vert q(z)\vert\,dz\leq 2jw_nCr^{k-\vert\alpha\vert},$$
where $w_n$ is the volume of the unit ball in $\mathbb{R}^n$ and $C$ is the constant in Lemma 1, which depends only on $n$ and $k$.\\
We see now that for $x,y$ in $B_j$ with $\vert y -x\vert \leq {1\over j}$, the last estimate implies 
$$ \left\{ \begin{array}{l} 
\vert f_{\alpha}(y)-D^{\alpha}p_{k-1}(x;y)\vert\leq M(n,k,j)\vert x-y\vert^{k-\vert\alpha\vert},\;\; \forall\; \vert \alpha\vert\leq k-1 \vspace{0.3cm} \\  \vert f_{\alpha}(x)\vert\leq j\;\;     
                                                \end{array}
\right.  $$
and by applying the Whitney Extension Theorem \ref{whi} we are done. Observe that if $g_j\in C^{k-1,1}$ is the extension of the restriction of $f$ to $B_j$ provided by Theorem \ref{whi}, we do not only have that $g_j$ and $f$ agree in $B_j$, but also the first $k-1$ derivatives of the smooth extension $g_j$ coincide on $B_j$ with the coefficients of the approximate polynomial: we have $f_{\alpha}(x)=D^{\alpha}g(x),\;\;\forall x\in B_j,\,\forall \vert\alpha\vert\leq k-1.$
\end{proof}

We will now present a variant of Lemma \ref{u} where we allow the exponents of the denominator $\vert y-x\vert$ of the approximate limits  to be real numbers, not necessarily integers. This change will allow us to get decompositions with $C^{k-1,t}$ functions, $t\in (0,1]$. 

\begin{lem}\label{uu}
Let $f:\mathbb{R}^n\rightarrow\mathbb{R}^m$ be a measurable function, $k$ a positive integer, $t\in (0,1]$ and $N$ a subset of $\mathbb{R}^n$. Suppose that 
\begin{equation}\label{aaa}
\mathrm{ap}\limsup_{y\rightarrow x}{\vert f(y)-p_{k-1}(x;y)\vert\over\vert y-x\vert^{k-1+t}}<+\infty\;\;\text{for all}\;x\in\mathbb{R}^n\setminus N.
\end{equation}
Then there exists a decomposition 
$$\mathbb{R}^n=\bigcup^{\infty}_{j=1}B_j\cup N$$
such that for each $j\in\mathbb{N}$ there exists $g_j\in C^{k-1,t}(\mathbb{R}^n;\mathbb{R}^m)$ with $f_{\alpha}(x)=D^{\alpha}g_j(x)$ for all $x\in B_j$ and $\vert \alpha\vert \leq k-1$.
\end{lem}
\begin{proof}
The proof is exactly the same as that of Lemma \ref{u}, until the point where we use the Whitney Extension Theorem \ref{whi}. In this case we have that for each $j\in\mathbb{N}$ and for all $x,y\in B_j$ with $\vert y-x\vert\leq {1\over j}$,
$$ \left\{ \begin{array}{l} 
\vert f_{\alpha}(y)-D^{\alpha}p_{k-1}(x;y)\vert\leq M(n,k,j)\vert x-y\vert^{k-1+t-\vert\alpha\vert},\;\; \forall\; \vert \alpha\vert\leq k-1 \vspace{0.3cm} \\  \vert f_{\alpha}(x)\vert\leq j.\;\;      
                                                \end{array}
\right.  $$
At this point we use Theorem \ref{whi} with $\omega(s)=s^{t}$ and we conclude similarly.
\end{proof}

These last theorems will be useful for our versions of the Morse-Sard theorem (Theorem \ref{i} and \ref{o}) where the exceptional sets are countably $(\mathcal{H}^s,s)$ rectifiable of certain class $C^k$ (see the statements of the results for details). However, in order to achieve a result where we are allowed to work with $s-$sigmafinite exceptional sets (Theorem \ref{p}), it will be necessary to have at our disposal a result as the next one. We use once again the ideas of Liu-Tai \cite{Liu}, together with those of Whitney \cite[Theorem 1]{Whitney2}.

\begin{lem}\label{uuu}
Let $f:\mathbb{R}^n\rightarrow\mathbb{R}^m$ be a Borel function, $k$ a positive integer and $s>0$. Suppose that $f$ is approximately differentiable of order $k$ at $\mathcal{H}^s-$almost every point $x\in B$, where $B\subseteq\mathbb{R}^n$ is a Borel set and $\mathcal{H}^s(B)<\infty$. Suppose also that the coefficients $f_{\alpha}(x)$, $|\alpha|\leq k$, are Borel functions. Then there exists a decomposition
$$B=\bigcup^{\infty}_{j=1}B_j\cup N,$$
where for each $j\in\mathbb{N}$ there exists $g_j\in C^{k}(\mathbb{R}^n;\mathbb{R}^m)$ with $f_{\alpha}(x)=D^{\alpha}g_j(x)$ for all $x\in B_j$, $\vert \alpha\vert \leq k$, and $\mathcal{H}^s(N)=0$.
\end{lem}
\begin{proof}
Without loss of generality let us suppose that $f$ is approximately differentiable of order $k$ for all $x\in B$. 

Let $\rho>0$ be as in the proof of Lemma \ref{ua}; recall that $\rho$ only depends on $n$. For each $\eta>0$, $i\in\mathbb{N}$, $x\in B$, define
$$W_{\eta}\left( x;i\right) :=B(x;{{1\over i}})\setminus \left\lbrace y\in B:\,|f(y)-p_k(x;y)|\leq \eta|y -x|^k\right\rbrace .$$
Since $f$ is Borel, these sets are Borel measurable. Consider the set
$$T=\left\lbrace (x,y)\in B\times B:\,|y-x|<{1\over i},\,|f(y)-p_k(x;y)|>\eta|y-x|^k\right\rbrace .$$
All the $f_{\alpha}$ are Borel functions, so $T$ is a Borel measurable set in $B\times B$. It is clear that $W_{\eta}(x;i)=\left\lbrace y\in B:\,(x,y)\in T\right\rbrace $, hence from Fubini's Theorem (see e.g.\,\cite[Proposition 5.1.2.]{Co}) it follows that $\mathcal{L}^n(W_{\eta}(x;i))$ is a Borel measurable function of $x$.

We know that for all $\eta>0$ and $x\in B$,
\begin{equation}\label{33}
\lim_{i\rightarrow\infty}{\mathcal{L}^n(W_{\eta}(x;i))\over\mathcal{L}^n(B(x;{1\over i}))}=0.
\end{equation}

Define
$$\phi_i(x):=\inf\left\lbrace \eta>0:\, \mathcal{L}^n(W_{\eta}(x;i))<{\rho\over 4}\left( {1\over i}\right) ^n\right\rbrace .$$
for each $i\in\mathbb{N}$ and $x\in B$. For fixed $x$, $\phi_i$ is decreasing in $\eta$ and continuous on the left. Thus
\begin{equation}\label{44}
\phi_i(x)\leq\eta\;\;\text{if and only if}\;\; \mathcal{L}^n(W_{\eta}(x;i))<{\rho\over 4}\left( {1\over i}\right)^n 
\end{equation}
and we have that $\phi_i(x)$ is a Borel measurable function.

From \eqref{33} and \eqref{44} it also follows that
$$\lim_{i\rightarrow\infty}\phi_i(x)=0\;\;\text{for every}\; x\in B.
$$

Now, with the goal of getting uniform convergence (up to a small enough set) in the previous limit, we want to apply Egorov's theorem for the measure $\mathcal{H}^s$. Notice that we are allowed to do so because the functions $\phi_i$ are Borel, the set $B$ is $\mathcal{H}^s-$finite, and $\mathcal{H}^s$ is a Borel measure. We thus obtain, for each $j\in\mathbb{N}$, a closed\footnote{Egorov's theorem in general would give us Borel sets $B_j$, but the Hausdorff measures $\mathcal{H}^s$ are Borel regular measures, so it is well known that for every Borel set $A$, if $\mathcal{H}^s(A)<\infty$ there exists for each $\varepsilon>0$ a closed set $C$ such that $C\subseteq A $ and $\mathcal{H}^s(A\setminus C)<\varepsilon$. This fact cannot be overlooked because our using Theorem \ref{whi2} forces us to work with closed sets.} set $B_j\subseteq B$ such that $\mathcal{H}^s(B\setminus B_j)<{1\over j}$ and $\lim_{i\rightarrow\infty}\phi_i(x)=0$ uniformly on $B_j$. 

Observe that $\mathcal{H}^s\left( B\setminus\bigcup^{\infty}_{j=1}B_j\right) =0$. Let us call $N=B\setminus\bigcup^{\infty}_{j=1}B_j$.\\

Now we just have to see that for each of these sets $B_j$ we can apply Theorem \ref{whi2} in order to get a $C^k$ extension to the whole space. 
Fix $j\in\mathbb{N}$ and a multi-index $|\alpha|\leq k$. We have to prove that for each $\varepsilon>0$, there exists $\delta>0$ such that
$$\vert f_{\alpha}(y)-D^{\alpha}p_k(x;y)\vert\leq \varepsilon\vert y-x\vert^{k-|\alpha|}\;\;\text{if}\; x,y\in B_j,\;\vert y-x\vert <\delta.$$

Let then $\varepsilon>0$. If $C>0$ denotes the constant of De Giorgi's Lemma \ref{ua}, choose $i_0\in\mathbb{N}$ such that
$$\vert \phi_i(x)\vert\leq {\varepsilon\over 2w_nC(1+\varepsilon)}:=\varepsilon_0\;\;\text{for all}\; x\in B_j\;\text{and all}\; i\geq i_0.$$
Also take $i_1\geq i_0$ sufficiently large such that $\left( 1+{1\over i_1}\right)^k\leq 1+\varepsilon$. 
Using \eqref{44} we have
$$\mathcal{L}^n\left( W_{\varepsilon_0}(x;i)\right) <{\rho\over 4}\left( {1\over i}\right)^n\;\;\text{for all}\; x\in B_j\;\text{and all}\; i\geq i_1.$$
Take $\delta<{1\over i_1}$ and $x,y\in B_j$ with $\vert y-x\vert <\delta$.
There exists $i_2\geq i_1$ such that ${1\over i_2+1}\leq |y-x|\leq {1\over i_2}$.
Now consider the set $S\left( x,y;i_2,\varepsilon_0\right)=\left[ B(x,{1\over{i_2}})\cap B(y,{1\over{i_2}})\right] \setminus\left[ W_{\varepsilon_0}(x;i_2)\cup W_{\varepsilon_0}(y;i_2)\right]$ analogously to Lemma \ref{u}. We have that 
$$\mathcal{L}^n\left( S\left( x,y;i_2,\varepsilon_0\right)\right)\geq\mathcal{L}^n\left( B\left( x,{1\over{i_2}}\right) \cap B\left( x+{y-x\over i_2|y-x|},{1\over{i_2}}\right)\right) -{\rho\over 2}\left( {1\over i_2}\right) ^n={\rho\over 2}\left( {1\over i_2}\right) ^n >0.$$
Then take $z\in S\left( x,y;i_2,\varepsilon_0\right)$ and observe that $|z-x|^{k-|\alpha|},|z-y|^{k-|\alpha}\leq (1+\varepsilon)|y-x|^{k-|\alpha|}$. Using De Giorgi's Lemma \ref{ua} in the same way as in the previous lemmas, we conclude that
$$\vert f_{\alpha}(y)-D^{\alpha}p_{k}(x;y)\vert\leq \varepsilon\vert x-y\vert^{k-\vert\alpha\vert}. $$
By applying the Whitney Extension Theorem \ref{whi2} the proof is complete.
\end{proof}

\section{A Morse-Sard Theorem for approximate differentiable functions}

Our aim is to prove a Morse-Sard theorem for functions that only are approximately differentiable of order $k$ or that have approximate $(k-1)-$Taylor polynomials on some sets. Consequently we will need to deal with weaker notions of derivatives and critical sets.
\begin{defn}
{\em Let $f:\mathbb{R}^n\rightarrow\mathbb{R}^m$ be a measurable function that is approximately differentiable of order $k$ at $x$, with unique approximate polynomial
$$p_{k}(x)=\sum_{\vert \alpha\vert\leq k}{f_{\alpha}(x)\over\alpha!}(y-x)^{\alpha}.$$
For each multi-index $\alpha$, $\vert\alpha\vert\leq k$, we define the {\em $\alpha-$th differential coefficient} of $f$ at $x$ as $f_{\alpha}(x)$. If $f$ has an approximate $(k-1)-$Taylor polynomial at $x$ we can only define the $\alpha-$th differential coefficient $f_{\alpha}(x)$ for $|\alpha|\leq k-1$.}
\end{defn}

Recall that we do not necessarily have $f_{\alpha}(x)=D^{\alpha}f(x)$ in any usual sense, and $D^{\alpha}f(x)$ may even not exist. 
 
From now on we will use the following notation
$$p_{k}(x;y)=\sum_{\vert \alpha\vert\leq k}{f_{\alpha}(x)\over\alpha!}(y-x)^{\alpha}=f(x)+F_1(x)(y-x)+\cdots {F_{k}f(x)\over (k)!}(y-x)^k.$$ 
where the $F_j(x)$ are the $j-$multilinear and symetric maps whose coefficients with respect to the standard basis of $\R^n$ are given by $f_{\alpha}(x)$, $|\alpha|=j$ ($j=1,\dots, k$). Again we stress that we do not necessarily have $F_j(x)=D^{j}f(x)$, and the latter may not exist. However, if a function is one time differentiable at $x$ in the usual sense, we do have $Df(x)=F_1(x)$.

We are now in a position to introduce a generalized notion of critical set.

\begin{defn} 
{\em If $f$ is approximately differentiable of order $1$ on some set of points, we define
$$C_f:=\left\lbrace x\in\mathbb{R}^n:\, F_1(x)\;\text{is defined and} \;\text{rank}\,(F_1(x))\;\text{is not maximum}\right\rbrace .$$ }
\end{defn}

\begin{rem}\label{hh}
{\em If a function $f$ only has an approximate $(0)-$Taylor polynomial at almost every point of $\mathbb{R}^n$ we can still define the set of critical points up to an $\mathcal{L}^n-$null set. According to Liu-Tai's result \cite[Theorem 1]{Liu}, $f$ is approximately differentiable of order $1$ almost everywhere, so we consider the coefficients of the linear part of the corresponding polynomial.}

\end{rem}

\begin{defn}[$(N_0)-$property] 
{\em Let $f:\mathbb{R}^n\rightarrow\mathbb{R}^m$ be measurable and let us suppose we have a notion of derivative for $f$, and hence a set of critical points $C_f$. We say that f satisfies the $(N_0)-$property with respect to the Hausdorff measure $\mathcal{H}^{s}$, $s\in(0,n]$, if and only if
$$E\subseteq C_f,\;\;\mathcal{H}^s(E)=0\Rightarrow\;\mathcal{L}^m(f(E))=0.$$}
\end{defn}

\begin{defn}[Luzin's $N-$property] 
{\em Let $f:\mathbb{R}^n\rightarrow\mathbb{R}^m$ be measurable. We say that $f$ satisfies Luzin's $N-$property with respect to the Hausdorff measure $\mathcal{H}^{s}$, $s>0$, if and only if
$$E\subseteq \mathbb{R}^n,\;\;\mathcal{H}^s(E)=0\Rightarrow\;\mathcal{H}^s(f(E))=0.$$
}
\end{defn}

The following theorem, due to Norton \cite[Theorem 2]{No}, will also be an important ingredient in our proofs of Theorem \ref{i}, \ref{o} and \ref{p}.

\begin{thm}[Norton]\label{nor}
Let $k$ be a positive integer, $t\in(0,1]$ and $f:\mathbb{R}^n\rightarrow\mathbb{R}^m$. 
\begin{enumerate}[(i)]
\item If $f\in C^{k,t}$ and $E\subseteq C_f$ is $\mathcal{H}^{k+t+m-1}-$null, then $\mathcal{L}^m(f(E))=0$. That is to say, $f$ has the $(N_0)-$property with respect to the measure $\mathcal{H}^{k+t+m-1}$.
\item If $f\in C^{k}$ and $E\subseteq C_f$ is $(k+m-1)-$sigmafinite, then $\mathcal{L}^m(f(E))=0$
\end{enumerate}
\end{thm}

We will also need to use Bates's version of the Morse-Sard Theorem for $C^{n-m,1}$ (see \cite[Theorem 2]{Ba}).
\begin{thm}[Bates]\label{bates}
Let $n$, $m$ be positive integers with $m\leq n$ and $f:\mathbb{R}^n\rightarrow\mathbb{R}^m$. If $f\in C^{n-m,1}(\mathbb{R}^n;\mathbb{R}^m)$, then the set of critical values of $f$ has $\mathcal{L}^m-$measure zero.
\end{thm}

The first of our main results is as follows. 

\begin{thm}\label{i}
Let $f:\mathbb{R}^n\rightarrow\mathbb{R}^m$, $m\leq n$. Suppose that 
\begin{enumerate}[(a)]
\item  $$\mathrm{ap} \limsup_{y\rightarrow x}{\vert f(y)-f(x)\vert\over \vert y-x\vert}<+\infty\;\;\text{for all} \;x\in\mathbb{R}^n\setminus N_0,$$
where $N_0$ is a countable set.
\item The set $$N_1 :=\{x\in\R^n: \mathrm{ap} \limsup_{y\rightarrow x}{\vert f(y)-f(x)-F_1(x)(y-x)\vert\over \vert y-x\vert^2}=+\infty\}$$ is $\mathcal{H}^m-$null.
\item For each $i=2,\dots,n-m$, the set
$$N_{i}:=\{x\in\R^n : \mathrm{ap} \limsup_{y\rightarrow x}{\left| f(y)-f(x)-\cdots-{F_i(x)\over i!}(y-x)^{i}\right|\over \vert y-x\vert^{i+1}}=+\infty\}$$ is countably $(\mathcal{H}^{i+m-1},i+m-1)$ rectifiable of class $C^i$.
\end{enumerate}
Then $\mathcal{L}^m(f(C_f))=0$.\\
The same is true if we replace $\R^n$ with an open subset $U$ of $\R^n$.
\end{thm}
[Observe that if $n=m$ we only have $(a)$ and if $n=m+1$ we only have $(a)$ and $(b)$.]
\begin{proof}
Note that $(a)$ tells us that $f$ is approximately continuous $\mathcal{L}^n-$almost everywhere so the function $f$ is measurable.

Let us also set some notation by writing each exceptional set $N_i$ ($i=1,\dots,n-m$) as
$$N_i=A_i\cup\bigcup^{\infty}_{k=1}A_{i,k},$$
where $\mathcal{H}^{i+m-1}(A_i)=0$ and each subset $A_{i,k}\subseteq\mathbb{R}^n$ is an $(i+m-1)-$dimensional submanifold of class $C^i$.

First of all let us show that condition $(a)$ implies Luzin's $N-$condition with respect to the Hausdorff measure $\mathcal{H}^s$, $s\in (0,m]$. In fact we will see that there exists a collection $\left\lbrace B_{0,j}\right\rbrace ^{\infty}_{j=1}$ with $\mathbb{R}^n=\bigcup^{\infty}_{j=1}B_{0,j}\cup N_0$ and such that each restriction $f|_{B_{0,j}}$ is locally Lipschitz with constant $2j$. Since $\mathcal{L}^m(f(N_0))=0$, this readily implies that the image of sets of $s-$Hausdorff measure zero ($s\leq m$) has $s-$Hausdorff measure zero.
The argument is again inspired by Liu-Tai's result (\cite[Theorem 1]{Liu}). Consider the sets $B_{0,j}$ as in the proof of Lemma \ref{u} above. We take $x,y\in B_{0,j}$, $\vert y-x\vert \leq {1\over j}$, $r=\vert y-x\vert$. Using that $\mathcal{L}^n(S(x,y;r,j))>0$ it is possible to take $z\in S(x,y;r,j)$ and then
$$\vert f(y)-f(x)\vert\leq \vert f(y)-f(z)\vert +\vert f(z)-f(x)\vert \leq j\vert z-y\vert +j\vert z-x\vert\leq 2j\vert y-x\vert.$$
Now it is clear that, for every $s\in (0,m]$, if $\mathcal{H}^{s}(A)=0$, $A\subseteq \mathbb{R}^n$, then $\mathcal{H}^s(f(A))=0$. Therefore the points where $C_f$ is not defined, which have $\mathcal{H}^m-$measure zero (recall that $f$ is approximately differentiable of order $1$ at $\mathcal{H}^m-$almost every point), have $\mathcal{L}^m-$null image.

Let us make a pause to comment on the special case $n=m$ (we only have condition $(a)$). In this case we also have the critical set of points defined up to a set of $\mathcal{L}^n-$measure zero (see Remark \ref{hh}). Moreover Liu-Tai's result \cite[Theorem 1]{Liu} asserts in particular that
$$\mathbb{R}^n=\bigcup^{\infty}_{j=1}D_j\cup M,$$
where $\mathcal{L}^n(M)=0$ and such that for each $j\in\mathbb{N}$ there is a function $g_j\in C^{1}(\mathbb{R}^n;\mathbb{R}^n)$ with $$D_{j}\subseteq\left\lbrace x\in\mathbb{R}^n:\, f(x)=g_j(x),\, F_1(x)=Dg_j(x)\right\rbrace .$$
Using the classical Morse-Sard theorem we have that for every $j\in\mathbb{N}$,
$$\mathcal{L}^n(f(C_f\cap D_j))=\mathcal{L}^n(f|_{D_{j}}(C_f\cap D_{j}))=\mathcal{L}^n(g_j(C_{g_j}\cap D_{j}))=0.$$
Consequently 
$$\mathcal{L}^n(f(C_f))\leq \sum^{\infty}_{j= 1}\mathcal{L}^n(f(C_f\cap D_j))+\mathcal{L}^n(f(M))=0.$$

$\blacktriangleright$ Step 1: Condition $(c)$ with $i=n-m$ allows us to use  Lemma \ref{u} and write
$$\mathbb{R}^n=\bigcup^{\infty}_{j=1}B_{n-m,j}\cup N_{n-m},$$
in such a way that for each $j\in\mathbb{N}$ there is a function $g_j\in C^{n-m,1}$ with $$B_{n-m,j}\subseteq\left\lbrace x\in\mathbb{R}^n:\, f(x)=g_j(x),\, F_1(x)=Dg_j(x)\right\rbrace .$$

We decompose $C_f$ as
$$C_f=\left( \bigcup^{\infty}_{j=1}C_f\cap B_{n-m,j}\right) \cup\left( C_f\cap N_{n-m}\right).$$ 
Using Bates's result (Theorem \ref{bates}) we have that for every $j\in\mathbb{N}$,
$$\mathcal{L}^m(f(C_f\cap B_{n-m,j}))=\mathcal{L}^m(f|_{B_{n-m,j}}(C_f\cap B_{n-m,j}))=\mathcal{L}^m(g_j(C_{g_j}\cap B_{n-m,j}))=0.$$
By the subadditivity of the Lebesgue measure we have thus reduced our problem to showing that $\mathcal{L}^m(f(C_f\cap N_{n-m}))=0$.\\

$\blacktriangleright$ Step 2: We now work with the condition $(c)$ but for the case $i=n-m-1$. By applying Lemma \ref{u} again, we obtain a decomposition
$$\mathbb{R}^n=\bigcup^{\infty}_{j=1}B_{n-m-1,j}\cup N_{n-m-1},$$
where for each $j\in\mathbb{N}$ there is a function $g_j\in C^{n-m-1,1}$ with \\
$$B_{n-m-1,j}\subseteq\left\lbrace x\in\mathbb{R}^n:\, f(x)=g_j(x),\, F_1(x)=Dg_j(x)\right\rbrace .$$
Recall that $N_{n-m}=A_{n-m}\cup\bigcup^{\infty}_{k=1}A_{n-m,k}$ where $\mathcal{H}^{n-1}(A_{n-m})=0$ and such that there exist maps 
$$\phi_{n-m,k}:\mathbb{R}^{n-1}\longrightarrow A_{n-m,k}\subseteq\mathbb{R}^n$$ of class $C^{n-m}$ for each $k\in\mathbb{N}$. 

We now write $C_f\cap N_{n-m}$ as

\begin{equation}\label{ww}
C_f\cap N_{n-m}=\left( \bigcup^{\infty}_{j=1}C_f\cap A_{n-m}\cap B_{n-m-1,j}\right) \cup\left( \bigcup^{\infty}_{j=1}\bigcup^{\infty}_{k=1}C_f\cap A_{n-m,k}\cap B_{n-m-1,j}\right) \cup\left( C_f\cap N_{n-m-1}\right) .
\end{equation}

Remember that $\mathcal{H}^{n-1}(A_{n-m})=0$, so by using Norton Theorem \ref{nor} $(i)$, for every $j\in\mathbb{N}$ we get
\begin{align*}
\mathcal{L}^m(f(C_f\cap A_{n-m}\cap B_{n-m-1,j}))&=\mathcal{L}^m(f|_{B_{n-m-1,j}}(C_f\cap A_{n-m}\cap B_{n-m-1,j}))=\\
&=\mathcal{L}^m(g_j(C_{g_j}\cap A_{n-m}\cap B_{n-m-1,j}))=0.
\end{align*}
Fix $j,k\in \mathbb{N}$. We now see that $\mathcal{L}^m(f(C_f\cap A_{n-m,k}\cap B_{n-m-1,j}))=0$. Indeed,

It is easy to check that $C_{g_j}\cap A_{n-m,k}\subseteq \phi_{n-m,k}(C_{g_j\circ \phi_{n-m,k}})$, so
\begin{align*}
\mathcal{L}^m(f(C_f\cap A_{n-m,k}\cap B_{n-m-1,j}))&=\mathcal{L}^m(g_j(C_{g_j}\cap A_{n-m,k}\cap B_{n-m-1,j}))\leq\\
&\leq \mathcal{L}^m(g_j(\phi_{n-m,k}(C_{g_j\circ\phi_{n-m,k}})\cap B_{n-m-1,j}))=\\
&=\mathcal{L}^m(g_j\circ\phi_{n-m,k}|_{\phi^{-1}_{n-m,k}(B_{n-m-1,j}\cap A_{n-m,k})}(C_{g_j\circ\phi_{n-m,k}}))=0,
\end{align*}
where in the last equality we have used using Bates's theorem \ref{bates}) applied to the function
$$g_j\circ\phi_{n-m,k}|_{\phi^{-1}_{n-m,k}(B_{n-m-1,j}\cap A_{n-m,k})}:\mathbb{R}^{n-1}\longrightarrow\mathbb{R}^m,$$
which is of class $C^{n-m-1,1}_{\textrm{loc}}(\mathbb{R}^{n-1};\mathbb{R}^m)$.

Therefore, by the subadditivity of the Lebesgue measure and \eqref{ww}, our problem boils down to checking that $\mathcal{L}^m(f(C_f\cap N_{n-m-1}))=0$.\\

$\blacktriangleright$ Final step: Reasoning in the same way for the cases $i=n-m-2,\dots,i=1$, we inductively arrive to the conclusion that it is enough to prove that $\mathcal{L}^m(f(C_f\cap N_1))=0$, which is now automatic using Luzin's $N-$condition with respect to the measure $\mathcal{H}^m$.
\end{proof}

\begin{rem}
{\em It is clear that the above proof can be adapted to get a similar result in which condition $(b)$ is dropped and condition $(c)$ now holds for $i=1, 2,\dots, n-m$. In principle this result is more general than Theorem \ref{i}. The reason for our statement of Theorem \ref{i} is that one of the typical applications of Morse-Sard-type theorems is ensuring that for almost every $y\in\R^m$ the set $f^{-1}(y)$ is regular enough (for instance, it is a $C^1$ manifold if $f$ is assumed to be $C^1$), a property that we would lose if we do not require condition $(b)$.}
\end{rem}

Let us consider now the simpler case where the exceptional sets $N_i$ are $\mathcal{H}^{i+m-1}-$null for $i=1,\dots,n-m$ (i.e $A_{i,k}=\emptyset$ for all $k\geq 1$). We establish an alternate version of the preceding result, in which we let the exponents of the denominators $\vert y-x\vert$ be smaller, and not integers. In return we must ask these limits to be finite in larger sets in order to achieve the Morse-Sard property. The arguments will be the same, but for our use of Lemma \ref{uu} instead of Lemma \ref{u}.

\begin{thm}\label{o}
Let $f:\mathbb{R}^n\rightarrow\mathbb{R}^m$, $m\leq n$. 

If $n>m+1$, for each $i=0,\dots,n-m-1$ choose numbers $s(i)\in(i,i+1]$ and suppose that
\begin{enumerate}[(a)]
\item For $i=0$, $$\mathrm{ap} \limsup_{y\rightarrow x}{\vert f(y)-f(x)\vert\over \vert y-x\vert^{s(0)}}<+\infty\;\;\textrm{ for all } x\in\mathbb{R}^n\setminus N_0,\;\text{where}\;N_0\;\text{is countable}.$$
\item For $i=1$, $$\mathrm{ap} \limsup_{y\rightarrow x}{\vert f(y)-f(x)-F_1(x)(y-x)\vert\over \vert y-x\vert^{s(1)}}<+\infty\;\;\text{for}\; \mathcal{H}^{s(0)m}-\text{almost every}\; x\in \mathbb{R}^n.$$
\item For each $i=2,\dots,n-m-1$,
$$\mathrm{ap} \limsup_{y\rightarrow x}{\left|f(y)-f(x)-\cdots-{F_i(x)\over i!}(y-x)^{i}\right|\over \vert y-x\vert^{s(i)}}<+\infty\;\;\text{for}\; \mathcal{H}^{s(i-1)+m-1}-\text{almost every}\; x\in \mathbb{R}^n.$$
(If $n=m+2$ we do not have this condition).
\item  $$\mathrm{ap} \limsup_{y\rightarrow x}{\left| f(y)-f(x)-\cdots-{F_{n-m}(x)\over (n-m)!}(y-x)^{n-m}\right|\over \vert y-x\vert^{n-m+1}}<+\infty\;\;\text{for}\; \mathcal{H}^{s(n-m-1)+m-1}-\text{almost every}\; x\in \mathbb{R}^n.$$
\end{enumerate}

If $n=m+1$, after choosing $s(0)\in (0,1]$, suppose only $(a)$ and $(d)$, where the exceptional set in $(d)$ must be $\mathcal{H}^{s(n-m-1)m}=\mathcal{H}^{s(0)m}-$null.

If $n=m$ suppose only that condition $(a)$ with $s(0)=1$ holds everywhere except perhaps on a countable set $N_0$.

Then  $\mathcal{L}^m(f(C_f))=0$.
\end{thm}
\begin{proof}
The case $n=m$ is exactly the same as in Theorem \ref{i}. Suppose then $n>m$.

We will first see that $(a)$ implies that $\mathcal{L}^{m}(f(A))=0$ for every set $A\subseteq \mathbb{R}^n$ with $\mathcal{H}^{s(0)m}(A)=0$. This will be in harmony with our definition of critical set of points (recall that $(b)$ implies that $F_1$ is defined up to a set of $\mathcal{H}^{s(0)m}-$measure zero).\\
Again we employ arguments similar to previous proofs. We take the same decomposition of $\mathbb{R}^n$ as in Lemma \ref{u}, except that in this case we set
$$W_{j}(x;r)=B(x,r)\setminus\left\lbrace y\in\mathbb{R}^n:\,\vert f(y)-f(x)\vert\leq j\vert y-x\vert^{s(0)}\right\rbrace ,\;\; x\in\mathbb{R}^n,\;r>0,\;j\in\mathbb{N}.$$
Hence $\mathbb{R}^n=\bigcup^{\infty}_{j=1}B_j$ and for $j\in\mathbb{N}$ we consider two different points $x,y\in B_j$ with $\vert y-x\vert\leq {1\over j}$, $r=\vert y-x\vert$. Using that $\mathcal{L}^n(S(x,y;r,j))>0$ we can take $z\in S(x,y;r,j)$ and then
$$\vert f(y)-f(x)\vert\leq \vert f(y)-f(z)\vert +\vert f(z)-f(x)\vert \leq j\vert z-y\vert^{s(0)} +j\vert z-x\vert^{s(0)}\leq 2j\vert y-x\vert^{s(0)}.$$
So we obtain
$$\vert f(y)-f(x)\vert^m\leq 2j\vert y-x\vert^{s(0)m}.$$
From this it can be automatically checked that $\mathcal{L}^m(f(A))=0$ for all $A\subseteq \mathbb{R}^n$ with $\mathcal{H}^{s(0)m}(A)=0$.

$\blacktriangleright$ Step 1: Condition $(d)$ allows us to use  Lemma \ref{u} and find a decomposition
$$\mathbb{R}^n=\bigcup^{\infty}_{j=1}B_j\cup N_{n-m},\;\;\mathcal{H}^{s(n-m-1)+m-1}(N_{n-m})=0,$$
such that for each $j\in\mathbb{N}$ there is a function $g_j\in C^{n-m,1}$ with $$B_j\subseteq\left\lbrace x\in\mathbb{R}^n:\, f(x)=g_j(x),\, F_1(x)=Dg_j(x)\right\rbrace .$$ Hence by using Bates's result (Theorem \ref{bates}) we have that for every $j\in\mathbb{N}$,
$$\mathcal{L}^m(f(C_f\cap B_j))=\mathcal{L}^m(f|_{B_j}(C_f\cap B_j))=\mathcal{L}^m(g_j(C_{g_j}\cap B_j))=0.$$
So we have reduced our problem to prove that all sets $E\subseteq C_f$ with $\mathcal{H}^{s(n-m-1)+m-1}-$measure zero satisfy $\mathcal{L}^m(f(E))=0$.\\

$\blacktriangleright$ Step 2: We work with the condition $(c)$ for the case $i=n-m-1$. We apply  Lemma \ref{uu} for the case $k-1+t=s(n-m-1)$ and $\mu=\mathcal{H}^{s(n-m-2)+m-1}$ to find a decomposition
$$\mathbb{R}^n=\bigcup^{\infty}_{j=1}B_j\cup N_{n-m-1},\;\;\mathcal{H}^{s(n-m-2)+m-1}(N_{n-m-1})=0,$$
where for each $j\in\mathbb{N}$ there is a function $g_j\in C^{n-m-1,s(n-m-1)-n+m+1}$ with 
$$B_j\subseteq\left\lbrace x\in\mathbb{R}^n:\, f(x)=g_j(x),\, F_1(x)=Dg_j(x)\right\rbrace .$$
We write a given set $E\subseteq C_f$, $\mathcal{H}^{s(n-m-1)+m-1}(E)=0$ as
$$E=\bigcup^{\infty}_{j=1}(E\cap B_j)\cup (N\cap E)$$
where $\mathcal{H}^{s(n-m-2)+m-1}(N)=0$. Now we use Norton's result (Theorem \ref{nor} $(i)$) and for every $j\in\mathbb{N}$,
$$\mathcal{L}^m(f(E\cap B_j))=\mathcal{L}^m(f|_{B_j}(E\cap B_j))=\mathcal{L}^m(g_j(C_{g_j}\cap E\cap B_j))=0.$$
Therefore we must only consider sets $E\subseteq C_f$ such that $\mathcal{H}^{s(n-m-2)+m-1}(E)=0$ and check that $\mathcal{L}^m(f(E))=0$.\\

$\blacktriangleright$ Final step: Reasoning in the same way for the cases $i=n-m-2,\dots,i=1$, we arrive to the conclusion that it is enough to prove that sets $E\subseteq C_f$ with $\mathcal{H}^{s(0)m}-$measure zero satisfy $\mathcal{L}^m(f(E))=0$. But this follows from $(a)$, as we have already seen.
\end{proof}
\begin{rem}
{\em If we choose $s(i)=i+1$ for each $i=0,\dots,n-m-1$ we get exactly Theorem \ref{i} in the particular case that $\mathcal{H}^{i+m-1}(N_i)=0$, $i=1,\dots,n-m$.}
\end{rem}
The difference between this theorem and the previous one (for the case that $\mathcal{H}^{i+m-1}(N_i)=0$, $i=1,\dots,n-m$) is that we are able to lower the exponents in the denominator. However we must ask in return that the sets where these properties hold are bigger than in Theorem \ref{i}. So we can say that  Theorem \ref{o} generalizes a particular Theorem \ref{i}, but a selection of numbers $s(j)$ make it nor stronger, neither weaker than any other possible choice.

It is a natural question to ask whether or not we can change our exceptional sets $N_i$ ($i=1,\dots,n-m$) to be $(i+m-1)-$sigmafinite. This is the purpose of our next main result, in which we must work with the stronger notion of approximate differentiablity instead of the property of having approximate $(k-1)-$Taylor polynomials. The result will generalize Theorem \ref{i} but in addition we will have to require that $f$ and its differential coefficients $f_{\alpha}$ are Borel functions in order to use Lemma \ref{uuu}.

\begin{thm}\label{p}
Let $f:\mathbb{R}^n\rightarrow\mathbb{R}^m$ be a Borel function, $m\leq n$. Suppose that 
\begin{enumerate}[(a)]
\item $\mathrm{ap} \limsup_{y\rightarrow x}{\vert f(y)-f(x)\vert\over \vert y-x\vert}<+\infty\;\;\text{for all} \;x\in\mathbb{R}^n\setminus N_0,$ where $N_0$ is a $0-$sigmafinite set (a countable set).
\item $f$ is approximately differentiable of order $1$ for all $x\in\mathbb{R}\setminus N_1$, where $N_1$ is $\mathcal{H}^m-$null.
\item For each $i=2,\dots,n-m$, $f$ is approximately differentiable of order $i$ for all $x\in\mathbb{R}\setminus N_i$, where $N_i$ is $(i+m-2)-$sigmafinite.
(If $n\leq m+1$ we do not have this condition).
\item $f$ has an approximate $(n-m)-$Taylor polynomial for all $x\in\mathbb{R}\setminus N_{n-m+1}$, where $N_{n-m+1}$ is $(n-1)-$sigmafinite.\footnote{Observe that if $n=m$ we only have $(a)$ (however this implies $(b)$ using Liu-Tai's result \cite[Theorem 1]{Liu}) and if $n=m+1$ we only have $(a)$, $(b)$ and $(d)$.}
\end{enumerate}
Suppose also that the coefficients $f_{\alpha}(x)$, $\vert\alpha\vert \leq n-m$, are Borel functions.\\
Then $\mathcal{L}^m(f(C_f))=0$.
\end{thm}

\begin{proof}
The case $n=m$ is exactly the same as in Theorem \ref{i} and Theorem \ref{o}.

Recall that condition $(a)$ gives us the Luzin's $N-$condition with respect to the Hausdorff measure $\mathcal{H}^m$. In particular the set of points where $C_f$ is not defined, which have $\mathcal{H}^m-$measure zero, have $\mathcal{L}^m-$null image.\\

$\blacktriangleright$ Step 1: Condition $(d)$ allows us to reduce our problem to showing that $\mathcal{L}^m(f(C_f\cap N_{n-m+1}))=0$ (here we use the same arguments as in Step 1 of Theorem \ref{i}).

Since the set $N_{n-m+1}$ is $(n-1)-$sigmafinte, by the subadditivity of the Lebesgue measure, without loss of generality we can assume that $\mathcal{H}^{n-1}(N_{n-m+1})<\infty$. Consequently we can focus on studying sets $A\subseteq C_f$ with $\mathcal{H}^{n-1}(A)<\infty$.\\

$\blacktriangleright$ Step 2: We now work with the condition $(c)$ but for the case $i=n-m$. 

Let us call $B=A\cap(\mathbb{R}^n\setminus N_{n-m})$. If we prove that $\mathcal{L}^m(f(B))=0$ it will only be needed to see if $\mathcal{L}^m(f(C_f\cap N_{n-m+1}\cap N_{n-m}))=0$, or what is the same, $\mathcal{L}^m(f(C_f\cap N_{n-m}))=0$ (note that we can suppose $N_{n-m}\subset N_{n-m+1}$).

We have that $f$ is approximately differentiable of order $n-m$ everywhere on $B$ and $\mathcal{H}^{n-1}(B)<\infty$. We can apply Lemma \ref{uuu} and write
\begin{equation}\label{pa}
B=\bigcup^{\infty}_{j=1}B_j\cup N,
\end{equation}
where for each $j\in\mathbb{N}$ there exists $g_j\in C^{n-m}(\mathbb{R}^n;\mathbb{R}^m)$ with $f_{\alpha}(x)=D^{\alpha}g_j(x)$ for all $x\in B_j$ $(\vert \alpha\vert \leq n-m)$, and $\mathcal{H}^{n-1}(N)=0$.

With the set $N$ we proceed as in Theorem \ref{i} when we had to deal with $\mathcal{H}^{n-1}-$null sets and we had to apply Lemma \ref{u} together with Norton's Theorem \ref{nor} $(i)$, and we conclude that $\mathcal{L}^m(f(N))=0$.

For the sets $B_j\subseteq C_f$ we use Norton's Theorem \ref{nor} $(ii)$ with $C^{n-m}$ regularity and we have
$$\mathcal{L}^m(f(B_j))=\mathcal{L}^m(g_j(B_{j}))=0.$$

Therefore, by the subadditivity of the Lebesgue measure and \eqref{pa}, our problem boils down to checking if $\mathcal{L}^m(f(C_f\cap N_{n-m}))=0$.\\

$\blacktriangleright$ Final step: Reasoning in the same way for the cases $i=n-m-2,\dots,i=1$, we inductively arrive to the conclusion that it is enough to prove that $\mathcal{L}^m(f(C_f\cap N_1))=0$, where $\mathcal{H}^m(N_1)=0$. But this is automatic by $(a)$.

\end{proof}

\section{Final considerations and examples}
A key point in the above arguments is obtaining a nice splitting of the space $\mathbb{R}^n$ into a countable union of sets (plus a small enough exceptional set) such that our function has enough regularity on each of those sets. Following the same strategy, there is another well known property that allows a similar decomposition.
\begin{defn}
{\em A measurable function $f:\mathbb{R}^n\rightarrow \mathbb{R}^m$ is said to have the Luzin property of order $k$ with respect to the measure $\mu$ if for every $\varepsilon>0$ there is a function $g\in C^k(\mathbb{R}^n;\mathbb{R}^m)$ such that $$\mu\left( \left\lbrace x\in\mathbb{R}^n:\,f(x)\neq g(x)\right\rbrace \right) <\varepsilon.$$}
\end{defn}
It is clear that for such a function there always exist decompositions of the form
$$\mathbb{R}^n=\bigcup^{\infty}_{j=1}B_j\cup N,$$
where for each $j\in\mathbb{N}$ there is a function $g_j\in C^{k}(\mathbb{R}^n;\mathbb{R}^m)$ with $f_{\alpha}(x)=D^{\alpha}g_j(x)$ for all $x\in B_j$ $(\vert \alpha\vert \leq k)$, and $\mu (N)=0$. Therefore, if, instead of approximate differentiability in our conditions of Theorem \ref{i}, Theorem \ref{o} and Theorem \ref{p}, we consider Luzin properties of the corresponding orders, with respect to the same Hausdorff measures, we may obtain the same conclusions.
For example the analogue of Theorem \ref{i} would be the following.

\begin{thm}\label{Luzin type MS theorem}
Let $f:\mathbb{R}^n\rightarrow\mathbb{R}^m$, $m< n$, be locally Lipschitz. Suppose that for each $i=2,\dots,n-m+1$, $f$ has the Luzin property of order $i$ with respect to the measure $\mathcal{H}^{i+m-2}$. Then we have $\mathcal{L}^m(f(C_f))=0$.
\end{thm}
For the proof we just mention that local Lipschitzness gives us the Luzin's $N-$property with respect to the measure $\mathcal{H}^m$, and that we must use the classical Morse-Sard Theorem instead of Bates's result. 

However, we cannot deduce Theorems \ref{i}, \ref{o} and \ref{p} from Theorem \ref{Luzin type MS theorem}, because, to the best of our knowledge, the problem whether an $\mathcal{H}^{s}$-a.e. approximately differentiable function of order $j$ must have the Luzin property of order $j$ (or a $C^{j-1,1}$ Luzin type property) with respect to the measure $\mathcal{H}^{s}_{\infty}$ or $\mathcal{H}^s$ is open. The proof of \cite{Liu} cannot be adapted to the measures $\mathcal{H}^{s}_{\infty}$ or $\mathcal{H}^s$ ($s<n$).\\

We will finally comment on three examples that illustrate how Theorem \ref{i} covers functions for which none of the previous Morse-Sard type results that exist in the literature can be applied to, and also how Theorem \ref{i} and Theorem \ref{p} are sharp in the following sense: for each $t\in (0, 1]$ we can always find a function $f:\mathbb{R}^n\rightarrow\mathbb{R}$ of class $C^{n-1}$ which has an approximate $(n-1)-$Taylor polynomial everywhere on $\mathbb{R}^n$ except on a set $N$ of Hausdorff dimension $n-1+t$, but which does not satisfy the Morse-Sard theorem.

\begin{enumerate}
\item We first note that Theorem \ref{i} is not weaker, nor stronger than the recent Bourgain-Korobkov-Kristensen generalizations \cite{Bo2,KK1,KK2} of the Morse-Sard theorem in the case of real-valued functions for the spaces $W^{n,1}_{loc}(\mathbb{R}^n;\mathbb{R})$ and $BV_{n,loc}(\mathbb{R}^n)$. The following example, taken from \cite[pag. 18]{Az1},
$$f(x,y)=\left\{ \begin{array}{lc} 
x^4\sin\left( {1\over x^2}\right)  &  \text{if}\; x\neq 0,  \\ \\  0 &\mbox{if}\; x=0,       
                                                \end{array}
\right. $$
shows that there are functions $f:\mathbb{R}^2\rightarrow\mathbb{R}$ satisfying the conditions of Theorem \ref{i} and such that $f\notin BV_{2}(\mathbb{R}^2)$. 

On the other hand there are functions $f:[0,1]\rightarrow\R$ which are in $W^{1,1}(\R;\R)$ (and therefore have the Morse-Sard property) and which do not satisfy assumption $(a)$ of Theorem \ref{i} because they 
$$
\lim_{y\to x}\frac{f(y)-f(x)}{y-x}=\infty
$$
for all $x$ in an uncountable set of measure zero. Those examples are well known, but we have not found any appropriate reference, so let us briefly recall a possible construction. Let us take the ternary Cantor set $C$ on $[0,1]$. For each $i=1,2,\dots$ choose a sequence of closed and disjoint intervals $\{I_{ij}\}_{j\in\N}$ such that $C\subseteq \bigcup^{\infty}_{j=1}I_{ij}$ and $\sum^{\infty}_{j=1}\ell(I_{ij})\leq (2/3)^i$. Define then
$$ f(x):=\int^{x}_{0}\sum^{\infty}_{i,j=1}\mathcal{X}_{I_{ij}}(t)\,dt=\sum^{\infty}_{i,j=1}\ell(I_{ij}\cap [0,x])\;,\;\;x\in [0,1].$$
Since the function $\sum^{\infty}_{i,j=1}\mathcal{X}_{I_{ij}}(t)$ is in $L^{1}[0,1]$, it is clear that $f$ is absolutely continuous. On the other hand, it is not difficult to check that $\lim_{h\rightarrow 0}{f(x+h)-f(x)\over h}=+\infty$ for every $x\in C$.

\item It is also worth noting that Theorem \ref{i} extends (\cite[Theorem 3.7]{Az2}): if $f\in C^{n-m}(\mathbb{R}^n;\mathbb{R}^m)$ is such that for all $x\in\mathbb{R}^n$ it has an approximate $(n-m-1)-$Taylor polynomial at $x$ then $f$ has the Morse-Sard property.
It is enough to take a function $f$ in the conditions of \cite[Theorem 3.7]{Az2} and for example change its value in all the points with rational coordinates (call this set $N$). This new function stops being of class $C^{n-m}$ but it still satisfies the assumptions of Theorem \ref{i}. Recall that we require condition $(a)$ in Theorem \ref{i} to hold everywhere except perhaps on a countable set, and that $\mathcal{H}^s(N)=0$ for all $s>0$.

\item We begin with some definitions. A subset $\gamma$ of $\mathbb{R}^n$ is an {\em arc} if it is the image of a continuous injection defined on the closed unit interval. For $x,y\in \gamma$, let $\gamma(x,y)$ denote the subarc of $\gamma$ lying between $x$ and $y$. An arc $\gamma$ is a {\em quasi-arc} if there is some $K>0$ such that for every $x,y\in\gamma$, $\gamma (x,y)$ is contained in some ball of radius $K|x-y|$.
A function $f$ is said to be critical on a set $A$ if $A\subseteq C_f$.
Let $A\subseteq \mathbb{R}^n$, $k\geq 1$ an integer number and $t\in (0,1)$. We say that $A$ is $(k+t)-${\em critical} if there exists a real-valued function $f\in C^{k,t}$ which is critical but not constant on $A$.\\
To provide a wide range of examples of functions $f:\mathbb{R}^n\rightarrow\mathbb{R}$ of class $C^{n-1}$ that have approximate $(n-1)-$Taylor polynomials everywhere on $\mathbb{R}^n$ except on at most a set $N$ with $\mathcal{H}^{n-1+t}(N)>0$, but that do not satisfy the Morse-Sard property, we state the following theorem from Norton (\cite[Theorem 2]{No2}).

\begin{thm}[Norton]
Let $k\geq 1$ be an integer number and $t\in (0,1)$. If $\gamma$ is a quasi-arc with $\mathcal{H}^{k+t}(\gamma)>0$, then $\gamma$ is $(k+t)-$critical.
\end{thm}
In the same paper Norton noted that such arcs ``are in plentiful supply (e.g. as Julia sets for certain rational maps in the plane)''. Hence, building such a quasi-arc $\gamma$ with $k=n-1$, $t\in (0,1)$ and $\mathcal{H}^{n-1+t}(\gamma)>0$, we can get a function $f:\mathbb{R}^n\rightarrow\mathbb{R}$ that is of class $C^{n-1,t}$ and does not satisfy the Morse-Sard theorem.
Note that we have
$$\text{ap}\limsup_{y\rightarrow x}{\left|f(y)-f(x)-\cdots -{D^{n-1}f(x)\over (n-1)!}(y-x)^{n-1}\right|\over |y-x|^n}<+\infty$$
for all $x\in\mathbb{R}^n\setminus \gamma$, since the construction of $f$ comes from an application of the Whitney extension theorem and consequently $f\in C^{\infty}(\mathbb{R}^n\setminus\gamma;\mathbb{R})$.\\ 
\end{enumerate}



\end{document}